\documentclass[12pt]{amsart}

\usepackage{amsmath,amsthm,amsfonts,amssymb,eucal, bbm}
\usepackage{scalerel}

\usepackage{color}

\renewcommand{\b}{\beta}

\newcommand{\g}{\gamma}
\newcommand{\G}{\Gamma}

\renewcommand{\d}{\delta}

\renewcommand{\l}{\lambda}
\renewcommand{\L}{\Lambda}
\newcommand{\z}{\zeta}
\renewcommand{\t}{\theta}

\newcommand{\p}{\partial}

\newcommand{\Om}{\Omega}

\newcommand{\oq}{\ {\raise 7pt\hbox{${\scriptstyle\circ}$}}
	\kern -7pt{
		\hbox{$Q$}}}

\newcommand{\R}{ \mathbb R}

\newcommand {\BS}{\mathbf S}

\newcommand {\bx}{\mathbf x}

\newcommand{\SfG}{{\sf{G}}}

\newcommand{\CH}{\mathcal H}

\newcommand{\CC}{\mathcal C}

\newcommand{\plainC}[1]{\textup{{\textsf{C}}}^{#1}}

\newcommand{\plainH}[1]{\textup{{\textsf{H}}}^{#1}}

\newcommand{\plainL}[1]{\textup{{\textsf{L}}}^{#1}}

\newcommand{\scale}{{\scaleto{1}{3pt}}}

\newcommand{\scalel}[1]
{{\scaleto{#1}{3pt}}}

\DeclareMathOperator{\iop}{{\sf Int}}

\DeclareMathOperator{\dc}{d}

\newtheorem{thm}{Theorem}[section]
\newtheorem{cor}[thm]{Corollary}
\newtheorem{lem}[thm]{Lemma}
\newtheorem{prop}[thm]{Proposition}

\theoremstyle{definition}

\newtheorem*{remark}{Remark}
\newtheorem{rem}[thm]{Remark}

\numberwithin{equation}{section}

\newcommand{\bee}{\begin{equation}}
	\newcommand{\ene}{\end{equation}}
\newcommand{\bees}{\begin{equation*}}
	\newcommand{\enes}{\end{equation*}}
\newcommand{\bes}{\begin{split}}
	\newcommand{\ens}{\end{split}}

\newcommand{\bet}{\begin{thm}}
	\newcommand{\ent}{\end{thm}}
\newcommand{\bel}{\begin{lem}}
	\newcommand{\enl}{\end{lem}}
\newcommand{\bec}{\begin{cor}}
	\newcommand{\enc}{\end{cor}}
\newcommand{\bep}{\begin{proof}}
	\newcommand{\enp}{\end{proof}}
\newcommand{\ber}{\begin{rem}}
	\newcommand{\enr}{\end{rem}}

\newcommand{\Z}{\mathbb Z}

\newcommand{\1}{\mathbbm 1}

\begin{document}
	\hoffset -4pc

\title
[ Density matrix]
{{Eigenvalue estimates for the one-particle density matrix}} 
\author{Alexander V. Sobolev}
\address{Department of Mathematics\\ University College London\\
	Gower Street\\ London\\ WC1E 6BT UK}
 \email{a.sobolev@ucl.ac.uk}

\begin{abstract} 
It is shown that the eigenvalues $\l_k, k=1, 2, \dots,$ of the one-particle density matrix 
satisfy the bound 
$\l_k\le C k^{-8/3}$ with a positive constant $C$.	 
\end{abstract}

\maketitle

\section{Introduction} 
Consider on $\plainL2(\R^{3N})$ the Schr\"odinger operator 
 \begin{align}\label{eq:ham}
H = \sum_{k=1}^N \bigg(-\Delta_k -  \frac{Z}{|x_k|}
\bigg) 
 + \sum_{1\le j< k\le N} \frac{1}{|x_j-x_k|},
\end{align}
describing an atom with $N$ electrons 
with coordinates $\bx = (x_1, x_2, \dots, x_N),\,x_k\in\R^3$, $k= 1, 2, \dots, N$, 
and a nucleus with charge $Z>0$. The notation $\Delta_k$ is used for 
the Laplacian w.r.t. the variable $x_k$. 
The operator $H$ acts on the Hilbert space $\plainL2(\R^{3N})$ and it is self-adjoint on the domain 
$D(H) =\plainH2(\R^{3N})$, since the potential in \eqref{eq:ham} 
is an infinitesimal perturbation 
relative to the unperturbed operator $-\Delta = - \sum_k \Delta_k$, 
see e.g. \cite[Theorem X.16]{ReedSimon2}. 
Note that we do not need to assume that the particles are fermions, i.e. that the 
underlying Hilbert space consists of anti-symmetric $\plainL2$-functions. Our results are 
not sensitive to such assumptions. 
Let $\psi = \psi(\bx)$, $\bx = (\hat\bx, x_N)$, $\hat\bx = (x_1, x_2, \dots, x_{N-1})$, 
be an eigenfunction of the operator $H$ with an eigenvalue $E\in\R$, i.e. $\psi\in D(H)$ and 
\begin{align*}
(H-E)\psi = 0.
\end{align*}
We define the one-particle density matrix as the function 
\begin{align}\label{eq:den}
\g(x, y) = \int\limits_{\R^{3N-3}} \overline{\psi(\hat\bx, x)} \psi(\hat\bx, y)\  
d\hat\bx,\quad (x,y)\in\R^3\times\R^3. 
\end{align} 
We do not discuss the importance of this object for multi-particle quantum mechanics and 
refer to the monograph 
\cite{RDM2000} for details. 
Our focus is on spectral properties of the self-adjoint non-negative 
operator $\G$ with the kernel $\g(x, y)$, which we 
call \textit{the one-electron density operator}.
Note that the operator 
$\G$ is represented as a product $\G = \Psi^*\Psi$ where 
$\Psi:\plainL2(\R^3)\to \plainL2(\R^{3N-3})$ is the operator with the kernel 
$\psi(\hat\bx, x)$. Since $\psi\in\plainL2(\R^{3N})$, the operator 
$\Psi$ is Hilbert-Schmidt, and hence $\G$ is trace class.  
Our objective is to investigate the decay of the eigenvalues $\l_k(\G)>0$, $k=1, 2, \dots,$ 
of the non-negative operator $\G$, labelled in descending order counting multiplicity. 
The significance of such information for quantum mechanical computations 
is discussed in the paper \cite{Fries2003}. In particular, it is shown in \cite{Fries2003} that 
$\G$ has infinite rank. 
Our main result is contained in Theorem \ref{thm:main} below. 
It establishes upper bounds on the decay of 
the eigenvalues $\l_k(\G)$, $k = 1, 2, \dots$ under the condition 
that $\psi$ decays exponentially as $|\bx|\to \infty$:
\begin{align}\label{eq:exp}
|\psi(\bx)|\lesssim e^{-\varkappa_\scalel{0} |\bx|_\scale},\ \bx\in\R^{3N}.
\end{align}  
Here $\varkappa_0 >0$ is a constant, and the notation ``$\lesssim$" means that the left-hand side 
is bounded from above by the right-hand side times 
some positive constant whose precise value is of 
no importance for us. This notation is used throughout the paper. 
Notice that instead of the standard Euclidean norm $|\bx|$ in \eqref{eq:exp} we have the 
$\ell_1$-norm which we denote by $|\bx|_1$. 
For the discrete eigenvalues, i.e. the ones 
below the bottom of the essential spectrum of $H$, the bound \eqref{eq:exp} follows from 
\cite{DHSV1978_79}. The exponential decay for eigenvalues away from the thresholds, including 
embedded ones, was studied 
in \cite{CombesThomas1973}, \cite{FH1982}. 
For more references and detailed discussion we quote \cite{SimonSelecta}. 

\begin{thm}\label{thm:main} 
Suppose that the eigenfunction $\psi$ satisfies the bound \eqref{eq:exp}. 
Let the function $\g(x, y)$, $(x, y)\in \R^3\times\R^3$, be defined by \eqref{eq:den}. 
Then the eigenvalues $\l_k(\G),k = 1, 2, \dots$, of the operator $\G$ satisfy the estimate 
\begin{align}\label{eq:main}
0<\l_k(\G)\lesssim k^{-\frac{8}{3}},\ \quad k = 1, 2, \dots,
\end{align}
with an implicit positive constant independent of $k$.
\end{thm}

\begin{remark}
\begin{enumerate}
\item 
The bound \eqref{eq:main} is sharp. This is confirmed by the asymptotic formula for the eigenvalues $\l_k(\G)$ 
which will be proved in a subsequent publication \cite{Sob2020}.  In fact, 
Theorem \ref{thm:main} or, more precisely, Theorem \ref{thm:psifull} can be regarded as a preparation for the 
sharp asymptotic formula in \cite{Sob2020}.
\item
Theorem \ref{thm:main} extends to the case of a molecule with several nuclei whose positions are fixed. 
The modifications are straightforward.
\item 
The choice of the norm $|\bx|_1$ instead of the 
Eucledian norm $|\bx|$ in the estimate \eqref{eq:exp} is made for computational convenience 
later in the proof.  
\end{enumerate}
\end{remark}

The strategy of the proof is quite straightforward: 
by virtue of the factorization $\G = \Psi^*\Psi$, mentioned a few lines earlier, 
we have $\l_k(\G) = s_k(\Psi)^2, k = 1, 2, \dots$, where 
$s_k(\Psi)$ are the singular values ($s$-values) of the operator $\Psi$. 
It is well-known that the rate of decay of singular values for integral operators 
depends on the smoothness of their kernels, and the appropriate estimates  
via suitable Sobolev norms can be found in the monograph \cite{BS1977} by M.S. Birman and M.Z. Solomyak. 
The regularity of $\psi$ has been well-studied in the literature. 
To begin with, according to the classical elliptic theory, due to the 
analyticity of the Coulomb potential $|x|^{-1}$ for $x\not = 0$, 
the function $\psi$ is real analytic away from the particle 
coalescence points. 
A more challenging problem is to understand the behaviour of $\psi$ at the coalescence points. 
The first result in this direction belongs to T. Kato \cite{Kato1957}, 
who showed that the function $\psi$ is Lipschitz. 
More detailed information 
on $\psi$ at the coalescence points was obtained, e.g. in \cite{FHOS2005}, \cite{FHOS2009}, \cite{HOSt1994}, 
and in the recent paper \cite{FS2018} by S. Fournais and T.\O. S\o rensen.  
The results of \cite{BS1977} and \cite{FS2018} are of crucial importance for the proof of Theorem \ref{thm:main}. 
A combination of the efficient bounds for the derivatives of the 
function $\psi$ obtained in \cite{FS2018}, and the estimates for 
the singular values in \cite{BS1977}, 
leads to the bound  
$s_k(\Psi)\lesssim k^{-4/3}$, and hence to \eqref{eq:main}. 
%

The plan of the paper is as follows.
In Sect. \ref{sect:reg} we list the facts that serve as ingredients of the proof. 
Although our aim is to prove the bound $s_k(\Psi)\lesssim k^{-4/3}$, 
in Sect. \ref{sect:prep} in Theorem \ref{thm:psifull} we state a bound for the operator 
$\Psi$ with weights which will be useful in the study of the spectral asymptotics for 
$\Psi$. The rest of Sect. \ref{sect:prep} 
 provides some preliminary estimates for auxiliary integral operators. 
These estimates are put together in Sect. \ref{sect:proofs} to complete the proof 
of Theorems \ref{thm:psifull} and \ref{thm:main}.

We conclude the introduction with some general notational conventions.  

\textit{Coordinates.} 
As mentioned earlier, we use the following standard notation for the coordinates: 
$\bx = (x_1, x_2, \dots, x_N)$,\ where $x_j\in \R^3$, $j = 1, 2, \dots, N$. 
The vector $\bx$ is usually represented in the form 
$\bx = (\hat\bx, x_N)$ with  
$\hat\bx = (x_1, x_2, \dots, x_{N-1})\in\R^{3N-3}$.  
In order to write 
formulas in a more compact and unified way, we sometimes use the notation 
$x_0 = 0$. 

In the space $\R^d, d\ge 1,$ 
the notation $|x|$ stands for the Euclidean norm, whereas $|x|_1$ denotes the $\ell_1$-norm.

\textit{Indicators.} For any set $\L\subset\R^d$ we denote by $\1_{\L}$ its indicator function (or indicator).
 
\textit{Derivatives.} 
Let $\mathbb N_0 = \mathbb N\cup\{0\}$.
If $x = (x', x'', x''')\in \R^3$ and $m = (m', m'', m''')\in \mathbb N_0^3$, then 
the derivative $\p_x^m$ is defined in the standard way:
\begin{align*}
\p_x^m = \p_{x'}^{m'}\p_{x''}^{m''}\p_{x'''}^{m'''}.
\end{align*} 
  
\textit{Bounds.} 
As explained earlier, for two non-negative numbers (or functions) 
$X$ and $Y$ depending on some parameters, 
we write $X\lesssim Y$ (or $Y\gtrsim X$) if $X\le C Y$ with 
some positive constant $C$ independent of those parameters.
 To avoid confusion we may comment on the nature of 
(implicit) constants in the bounds. 

\section{Ingredients of the proof}\label{sect:reg}
   
In this section we 
list three ingredients of the proof of the main Theorem 
\ref{thm:main}.

\subsection{Regularity of the eigenfunction} 
We need some efficient bounds for 
the derivatives of the eigenfunction away from the coalescence points, obtained by 
S. Fournais and T.\O. S\o rensen in \cite{FS2018}. 
Let 
\begin{align*}
\dc(\hat\bx, x) = \min\{|x|, |x-x_j|, \ j = 1, 2, \dots, N-1\}.
\end{align*}
The following proposition is a consequence of \cite[Corollary 1.3]{FS2018}:

\begin{prop}\label{prop:FS}
Assume that $\psi$ satisfies 
\eqref{eq:exp}. 
Then for all multi-indices $m\in\mathbb N_0^3$, $|m|_1\ge 1$, we have
\begin{align}\label{eq:FS}
|\p_x^m \psi(\hat\bx, x)|
\lesssim\dc(\hat\bx, x)^{1-l
} e^{-\varkappa_{\scalel{l}} {|\bx|_\scale}}, \ l = |m|_1,
\end{align}
with some $\varkappa_l >0$.
\end{prop}

The precise values of the constants $\varkappa_l>0$ are insignificant for us, 
and therefore we may assume that  
\begin{align}\label{eq:monotone}
\varkappa_0 = \varkappa_1\ge \varkappa_2\ge \dots >0. 
\end{align}
Let us rewrite the bounds \eqref{eq:FS} 
using the notation $x_0 = 0$.
With this convention, we have 
\begin{align*}
\dc(\hat\bx, x) = \min\{|x-x_j|, \ j = 0, 1, 2, \dots, N-1\}
\end{align*}
and 
\begin{align*}
\dc(\hat\bx, x)^{-1}\le \sum_{0\le j\le N-1} |x-x_j|^{-1}.
\end{align*} 
Therefore \eqref{eq:exp} and \eqref{eq:FS} imply that
\begin{align}\label{eq:FS1}
|\p_x^m \psi(\hat\bx, x)|\lesssim  e^{-\varkappa_{\scalel{l}} |\bx|_\scale}
\bigg(1+\sum_{0\le j\le N-1} |x-x_j|^{1-l}\bigg),\quad \ l = |m|_1,
\end{align}
for all $m\in\mathbb N_0^3$.

\subsection{Compact operators} 
Our main reference for compact operators is the book \cite{BS}. 
Let $\CH$ and $\mathcal G$ be separable Hilbert spaces.  
Let $T:\CH\to\mathcal G$ be a compact operator. 
If $\CH = \mathcal G$ and $T=T^*\ge 0$, then $\l_k(T)$, $k= 1, 2, \dots$, 
denote the positive eigenvalues of $T$ 
numbered in descending order counting multiplicity. 
For arbitrary spaces $\CH$, $\mathcal G$ and compact $T$, by $s_k(T) >0$, 
$k= 1, 2, \dots$, we denote the singular values of 
$T$ defined by $s_k(T)^2 = \l_k(T^*T) = \l_k(TT^*)$.  
Note the useful inequality 
\begin{align}\label{eq:2k}
s_{2k}(T_1+T_2)\le s_{2k-1}(T_1+T_2)\le s_k(T_1) + s_k(T_2),
\end{align}
which holds for any two compact $T_1, T_2$, see \cite[Formula (11.1.14)]{BS}. 
We classify compact operators by the rate of decay of their singular values. 
If $s_k(T)\lesssim k^{-1/p}, k = 1, 2, \dots$, 
with some $p >0$, then we say that $T\in \BS_{p, \infty}$ and denote
\begin{align}\label{eq:quasi}
\| T\|_{p, \infty} = \sup_k k^{\frac{1}{p}} s_k(T).
\end{align}
The class $\BS_{p, \infty}$ is a complete linear space with the quasi-norm $\|T\|_{p, \infty}$, see 
\cite[\S 11.6]{BS}.
For $p\in (0, 1)$ the quasi-norm satisfies the following ``triangle" inequality for  
operators $T_j\in\BS_{p, \infty}$, $j = 1, 2, \dots$:
\begin{align}\label{eq:triangle}
\big\|\sum_{j} T_j\big\|_{p, \infty}^p\le (1-p)^{-1}
\sum_j \|T_j\|_{p, \infty}^p,
\end{align}
see 
\cite[Lemmata 7.5, 7.6]{AJPR2002}, \cite[\S 1]{BS1977} and references therein. 
For the case $p >1$ see \cite[\S 11.6]{BS}, but we do not need it in what follows. 
 
For $T\in \BS_{p, \infty}$ the following number is finite:
\begin{align}\label{eq:limsupinf}
\SfG_p(T) = \big(\limsup_{k\to\infty} k^{\frac{1}{p}}s_k(T)\big)^{p},
\end{align}
and it clearly satisfies the inequality
\begin{align}\label{eq:estim}
\SfG_p(T)\le \|T\|_{p, \infty}^p.
\end{align}
More precisely, let $\BS_{p, \infty}^\circ\subset\BS_{p, \infty}$ be the closed subspace of all 
operators $R\in\BS_{p, \infty}$ with $\SfG_p(R) = 0$.
As explained in \cite[Theorem 11.6.10]{BS}, 
\begin{align}\label{eq:inf}
\SfG_p(T) = \inf_{R\in \BS_{p, \infty}^\circ}\| T+R\|_{p, \infty}^p.
\end{align}
The functional $\SfG_p(T)$, $p <1$, also satisfies the inequality of the type \eqref{eq:triangle}:

\begin{lem}\label{lem:triangleg}
Suppose that $T_j\in \BS_{p, \infty}$, $j = 1, 2, \dots, $ with some $p <1$ and that 
\begin{align}\label{eq:conv}
\sum_{j} \|T_j\|_{p, \infty}^p<\infty.
\end{align} 
Then 
\begin{align}\label{eq:triangleg}
\SfG_p\big(\sum_{j} T_j\big)\le (1-p)^{-1}
\sum_j \SfG_p (T_j).
\end{align}
\end{lem}

\begin{proof}  
By \eqref{eq:triangle} the operator $T = \sum_j T_j$ belongs to $\BS_{p, \infty}$, so that the left-hand side is finite. Furthermore, 
due to \eqref{eq:estim} and to the condition \eqref{eq:conv} 
the right-hand side of \eqref{eq:triangleg} is finite as well. 
Fix an $\varepsilon>0$ and pick $N$ such that 
\begin{align*}
\sum_{j=N+1}^\infty \|T_j\|_{p, \infty}^p<\varepsilon.
\end{align*}
Then by \eqref{eq:inf} and \eqref{eq:triangle}, 
for any $R_j\in \BS_{p, \infty}^\circ$, $j = 1, 2, \dots, N$, 
we have the estimate
\begin{align*}
\SfG_p(T)\le &\ \bigg\| \sum_{j=1}^N (T_j+R_j) + \sum_{j=N+1}^\infty T_j \bigg\|_p^p\\
\le &\ (1-p)^{-1}\bigg( \sum_{j=1}^N \|T_j+R_j\|_{p, \infty}^p + \varepsilon
\bigg).
\end{align*}
Minimizing the right-hand side over $R_j, j = 1, 2, \dots, N$, by 
\eqref{eq:inf} we get the estimate 
\begin{align*}
\SfG_p(T)\le &\ (1-p)^{-1}\bigg(\sum_{j=1}^N\SfG_p(T_j) + \varepsilon\bigg)\\
\le &\ (1-p)^{-1}\bigg(\sum_{j=1}^\infty \SfG_p(T_j) + \varepsilon\bigg).
\end{align*}
Since $\varepsilon>0$ is arbitrary, we obtain \eqref{eq:triangleg}.
\end{proof}

\subsection{Singular values of integral operators}
The final ingredient of the proof is the result due to M.S. Birman and M.Z. Solomyak, 
investigating the 
membership of integral operators in the class $\BS_{p, \infty}$ with some $p>0$. 
For estimates of the 
singular values we rely on \cite[Proposition 2.1]{BS1977}, see also 
\cite[Theorem 11.8.4]{BS},  
which we state here in a form convenient for 
our purposes. 
Let $\CC = (0, 1)^d\subset\R^d, d\ge 1$, be the unit cube. 
%

\begin{prop}\label{prop:BS}
Let $T_{ba}: \plainL2(\CC)\to\plainL2(\R^n)$, 
be the integral operator of the form
\begin{align*}
(T_{ba}u)(t) = b(t) \int_{\CC} T(t, x) a(x) u(x)\,dx,
\end{align*}
where $a\in\plainL2(\CC)$, $b\in\plainL2_{\tiny{\rm loc}}(\R^n)$, and 
the kernel $T(t, x)$, $t\in\R^n$, $x\in\CC$, is such that 
$T(t, \ \cdot\ )\in \plainH{l}(\CC)$ with some $l = 1, 2, \dots$, $2l >d$,
a.e. $t\in\R^n$.  
Then 
\begin{align*}
s_k(T_{ba})\lesssim k^{-\frac{1}{2}-\frac{l}{d}} 
\biggl[\int_{\R^n}  \|T(t, \ \cdot\ )\|_{\plainH{l}}^2\, |b(t)|^2\, dt\biggr]^{\frac{1}{2}} 
\|a\|_{\plainL2(\CC)},
\end{align*}
$k = 1, 2, \dots$, 
with some implicit constant independent of the kernel $T$, weights $a, b$ and the index $k$. In other words, 
$T_{ba}\in\BS_{q, \infty}$ with 
\begin{align*}
\frac{1}{q} = \frac{1}{2}+\frac{l}{d},
\end{align*}
and
\begin{align*}
\|T_{ba}\|_{q, \infty}\lesssim \biggl[\int_{\R^n}  
\|T(t,\ \cdot\ )\|_{\plainH{l}}^2\, |b(t)|^2 \,dt \biggr]^{\frac{1}{2}}
\|a\|_{\plainL2(\CC)}.
\end{align*}
\end{prop}   

It is straightforward to check that 
if one replaces the cube $\CC$ with its translate $\CC_n = \CC+n, n\in \Z^d$, then 
the bounds of Proposition \ref{prop:BS} still hold with implicit constants independent of $n$.

\section{Preliminary estimates}
\label{sect:prep}
   
\subsection{The weighted operator $\Psi$}
Represent the operator $\G$ as the product $\G = \Psi^*\Psi$,   
 where $\Psi: \plainL2(\R^3)\to\plainL2(\R^{3N-3})$ is defined by 
 \begin{align*}
 (\Psi u)(\hat\bx) = \int_{\R^3} \psi(\hat\bx, x) u(x) d x, u\in\plainL2(\R^3).
 \end{align*}
Since $\psi\in \plainL2(\R^{3N})$, this operator is Hilbert-Schmidt. 
As explained in the Introduction, in order to prove 
\eqref{eq:main} it suffices to show that $s_k(\Psi)\lesssim k^{-4/3}$, $k = 1, 2, \dots$, 
i.e. that $\Psi\in\BS_{3/4, \infty}$. 
For future use, we obtain an estimate for the operator $b \Psi a$ with weights $a$ and $b$. 
In order to describe these weights, denote $\CC_n = (0, 1)^3 + n$,\ $n\in\mathbb Z^3$. 
Let $\varkappa_l>0$ be the constants in the exponential bounds \eqref{eq:exp} and \eqref{eq:FS}. 
We assume that the weight  
$a\in\plainL2_{\textup{\tiny loc}}(\R^3)$ 
is such that 
\begin{align}\label{eq:Sq}
S_q^{(l)}(a) = \bigg[\sum_{n\in\mathbb Z^3} e^{- q\varkappa_{\scalel{l}}|n|_\scale}
\|a\|_{\plainL2(\CC_n)}^q\bigg]^{\frac{1}{q}}<\infty,\quad q = \frac{3}{4},
\end{align}
and that $b\in\plainL\infty(\R^{3N-3})$, so that 
\begin{align}\label{eq:mb}
M^{(l)}(b) = \biggl[\int_{\R^{3N-3}} 
|b(\hat\bx)|^2 e^{-2\varkappa_{\scalel{l}}|\hat\bx|_\scale} 
d\hat\bx\biggr]^{\frac{1}{2}}<\infty,\ \forall l = 1, 2, \dots.
\end{align}
Recall that the functional $\SfG_p$ is defined in \eqref{eq:limsupinf}. 
Our objective is to prove the following theorem. 

\begin{thm}\label{thm:psifull}
Let $b\in \plainL\infty(\R^{3N-3})$ and let $a\in\plainL2_{\textup{\tiny loc}}(\R^3)$ be such that  
$S_{3/4}^{(4)}(a)<\infty$.  
Then  $ b \Psi a\in \BS_{3/4, \infty}$ and 
\begin{align}
\|b\Psi a\|_{3/4, \infty}
\lesssim  &\ \|b\|_{\plainL\infty} S_{3/4}^{(3)}(a),\label{eq:psifullest}\\
 \SfG_{3/4}(b\Psi a)
\lesssim &\ \big( M^{(4)}(b) S_{3/4}^{(4)}(a)\big)^{\frac{3}{4}}.\label{eq:psifull}
\end{align}
\end{thm}

For $a = 1$ and $b = 1$ this theorem implies that $s_k(\Psi)\lesssim k^{-4/3}$, and hence 
$\l_k(\G) = s_k(\Psi)^2\lesssim k^{-8/3}$, thereby proving Theorem \ref{thm:main}.

\vskip 0.2cm

The plan of the proof is as follows. We study first the operators $\Psi_n = \Psi\1_{\CC_n}$, $n\in\Z^3$.
For each fixed $n$ the operator $\Psi_n$ is split in the sum of several 
operators depending on two parameters: $\d>0$ and $\varepsilon>0$, 
whose singular values are estimated in different ways. 
None of these estimates is sharp, but in the end, when collecting all the estimates together in Sect. 
\ref{sect:proofs}, we get the sharp bound \eqref{eq:psifull} by making a clever choice of the parameters 
$\d$ and $\varepsilon$.

For convenience we introduce the notation $\iop(T):\plainL2(\R^3)\to \plainL2(\R^{3N-3})$  
for the integral operator with the kernel $T(\hat\bx, x)$. 
Whenever we consider the operators $b\iop(\ \cdot\ )\1_{\CC_n}a$ with weights $a, b$, 
the constants in all the bounds are independent on the weights or on the parameter $n\in\mathbb Z^3$.

Recall also that we use the notation $x_0 = 0$. The symbol $\sum_j$ (resp. $\prod_j$) 
assumes summation (resp. product) over all $j = 0, 1, \dots, N-1$.

\subsection{Partition of $\Psi_n$: step 1}
The first step is to estimate the contribution of the domain on which the variables $x_j$, $j=0, 1, 2, \dots, N-1,$ 
are close to each other. 
Fix a $\d>0$ 
 and denote 
\begin{align*}
\Om^{(\d)} = &\ \bigcap_{0\le l<s\le  N-1}\{\hat\bx\in \R^{3N-3}: |x_l-x_s|> 4\d\}.
\end{align*}
The indicator of this set is denoted by $\chi^{(\d)}$, i.e.
\begin{align}\label{eq:chidel}
\chi^{(\d)}(\hat\bx) =  \1_{\Om^{(\d)}}(\hat\bx) = \prod_{0\le l<s\le N-1}\1_{\{|x_l-x_s|>4\d\}}(\hat\bx).
\end{align}
Represent $\psi$ as follows:
\begin{align}\label{eq:split1}
\psi = &\ \psi_1^{(\d)} + \psi_2^{(\d)},\\
\psi_1^{(\d)}(\hat\bx, x) = &\ \psi(\hat\bx, x) 
\chi^{(\d)}(\hat\bx),\notag\\
 \psi_2^{(\d)}(\hat\bx, x)
 = & \ \psi(\hat\bx, x) - \psi_1^{(\d)}(\hat\bx, x)
 = \psi(\hat\bx, x)\big(1-\chi^{(\d)}(\hat\bx)\big).\notag
\end{align}
It follows from \eqref{eq:FS1} that 
 \begin{align}
|\p_x^m 
\psi^{(\d)}_2(\hat\bx, x)| 
\lesssim 
e^{-\varkappa_{|m|_\scale}|\bx|_\scale}\bigg(1+\sum_j |x-x_j|^{1-|m|_\scale}\bigg)
\sum_{0\le l<s\le N-1} \1_{\{|x_l-x_s|<4\d\}}(\hat\bx),\label{eq:psidel2}
 \end{align}
for all $m\in\mathbb N_0^3$, with an implicit constant independent of $\d>0$. 
The operator $\iop(\psi_{2}^{(\d)})$ is considered with the weight $b = 1$ and arbitrary 
$a\in\plainL2(\CC_n)$.

In the next lemma and further on we use the straightforward inequality
\begin{align}\label{eq:expon}
\max_{x\in \CC_n}e^{-\varkappa_{\scalel{l}}|\bx|_\scale}
\le e^{3\varkappa_{\scalel{l}}}e^{-\varkappa_{\scalel{l}}|\hat\bx|_\scale} e^{-\varkappa_{\scalel{l}}|n|_\scale}.
\end{align}

\begin{lem}\label{lem:psidel}
The operator $\iop(\psi_2^{(\d)}) a\1_{\CC_n}$ belongs to $\BS_{6/7, \infty}$ and 
\begin{align}\label{eq:psidelcube}
\| \iop(\psi^{(\d)}_2)\, a\, \1_{\CC_n}\|_{6/7, \infty}
\lesssim e^{-\varkappa_{\scalel{2}}|n|_\scale} \d^{\frac{3}{2}}\|a\|_{\plainL2(\CC_n)},
\end{align}
for all $n\in\mathbb Z^3$ and all $\d>0$.
\end{lem}

\begin{proof} 
According to \eqref{eq:psidel2} and \eqref{eq:expon}, 
$\psi^{(\d)}_2(\hat\bx, \ \cdot\ )\in \plainH2(\CC_n)$ 
for a.e. $\hat\bx\in\R^{3N-3}$ and 
\begin{align*}
e^{2\varkappa_{\scalel{2}}|\hat\bx|_\scale}\|\psi^{(\d)}_2(\hat\bx,\ \cdot\ )\|_{\plainH2}^2 
\lesssim &\ e^{-2\varkappa_{\scalel{2}}|n|_\scale}
\int_{\CC_n} \bigg(1+\sum_{0\le j\le N-1} |x-x_j|^{-2}\bigg) dx\\ 
&\ \qquad\qquad\qquad \times\sum_{0\le l<s\le N-1} 
 \1_{\{|x_l-x_s|<4\d\}}(\hat\bx) \\[0.2cm]
\lesssim &\ e^{-2\varkappa_{\scalel{2}} |n|_\scale} 
\sum_{0\le l<s\le N-1} 
 \1_{\{|x_l-x_s|<4\d\}}(\hat\bx).
\end{align*}   
Using Proposition \ref{prop:BS} with $l=2, d=3$ (so that $2l >d$), we get 
that the operator on the left-hand side of \eqref{eq:psidelcube} belongs 
to $\BS_{q, \infty}$ with $q = 6/7$ and 
\begin{align*}
\|\iop(\psi^{(\d)}_2)\, a\,\1_{\CC_n}\|_{6/7, \infty}
\lesssim &\  \bigg[\int_{\R^{3N-3}} 
\|\psi^{(\d)}_2(\hat\bx, \ \cdot\ )\|_{\plainH2}^2\,d\hat\bx \biggl]^{\frac{1}{2}}\, \|a\|_{\plainL2(\CC_n)}
\notag\\
\lesssim &\ e^{-\varkappa_{\scalel{2}} |n|_\scale}\bigg[\int_{\R^{3N-3}}  
e^{-2\varkappa_{\scalel{2}}|\hat\bx|_\scale} \sum_{0\le l < s\le N-1}\1_{\{|x_l-x_s|<4\d\}}(\hat\bx)\,
d\hat\bx 
\bigg]^{\frac{1}{2}}\,\|a\|_{\plainL2(\CC_n)} \notag\\
\lesssim &\ e^{-\varkappa_{\scalel{2}}|n|_\scale} \d^{\frac{3}{2}}\,\|a\|_{\plainL2(\CC_n)},
\end{align*}
which gives \eqref{eq:psidelcube}.
\end{proof}

To study the kernel $\psi_1^{(\d)}$, 
we separate the contribution from the values of $x$ that are ``far" from $x_j$'s, $j=0, 1, \dots, N-1$.
Let $\t\in\plainC\infty_0(\R)$ be a function such that $0\le \t\le 1$ and 
\begin{align*}
\t(t) = 0,\quad \textup{if}\quad |t|>2;\ \quad
\t(t) = 1,\quad \textup{if}\quad |t|<1. \ 
\end{align*} 
Denote $\z(t) = 1 - \t(t)$. Observe that for any $\nu >0$, 
\begin{align}
\big|\p_x^m \t\big(|x|\nu^{-1}\big)\big|
\lesssim &\  \1_{\{|x|<2\nu\}} + \nu^{-|m|_\scale} \1_{\{\nu<|x|<2\nu\}}
\lesssim 
\nu^{-|m|_\scale} \1_{\{|x|<2\nu\}},\label{eq:nu}\\
\big|\p_x^m \z\big(|x|\nu^{-1}\big)\big|
\lesssim  &\ \1_{\{|x|>\nu\}} + \nu^{-|m|_\scale} \1_{\{\nu <|x|<2\nu\}}
\lesssim \nu^{-|m|_\scale} \1_{\{|x|> \nu\}},\notag
\end{align}
for all $m\in \mathbb N_0^3$. Consequently, 
\begin{align}\label{eq:cutoff}
\big|\p_x^m \t\big(|x|\nu^{-1}\big)\big|
\lesssim  |x|^{-|m|_\scale}\1_{\{|x|<2\nu\}},\quad
\big|\p_x^m \z\big(|x|\nu^{-1}\big)\big|
\lesssim   |x|^{-|m|_\scale}\1_{\{|x|>\nu\}}, \quad m\in\mathbb N_0^3,
\end{align}
uniformly in $\nu >0$.

In what follows we consider separately 
the following components of $\psi_1^{(\d)}$:
\begin{align}\label{eq:split2}
\psi_1^{(\d)} = &\ \psi_{11}^{(\d)} + \psi_{12}^{(\d)},\\
\psi_{11}^{(\d)}(\hat\bx, x) = &\  
\sum_{j} \t\big(|x-x_j|\d^{-1}\big) \psi_1^{(\d)}(\hat\bx, x),\notag\\
\psi^{(\d)}_{12}(\hat\bx, x) = &\ 
\big[1 - \sum_{j} \t\big(|x-x_j|\d^{-1}\big)\big] \psi_1^{(\d)}(\hat\bx, x).\notag
\end{align}
%
In view of the definition of $\chi^{(\d)}$, see \eqref{eq:chidel}, we have
\begin{align*}
\big[1 - \sum_{j} \t\big(|x-x_j|\d^{-1}\big)\big] \chi^{(\d)}(\hat\bx)
= \prod_j \z\big(|x-x_j|\d^{-1}\big) \, \chi^{(\d)}(\hat\bx),
\end{align*} 
so that 
\begin{align*}
\psi_{12}^{(\d)}(\hat\bx, x) = \psi_1^{(\d)}(\hat\bx, x) \prod_j \z\big(|x-x_j|\d^{-1}\big). 
\end{align*}
Estimate the derivatives of this function. First observe that in view of \eqref{eq:cutoff} we have  
\begin{align*}
\big|\p_x^m\prod_j \z\big(|x-x_j|\d^{-1}\big)\big|
\lesssim \bigg(\sum_j |x-x_j|^{-|m|_\scale}\bigg)
\prod_j 
\1_{\{|x-x_j|>\d\}}(\hat\bx, x),\quad m\in\mathbb N_0^3. 
\end{align*}
Together with \eqref{eq:FS1} this gives
\begin{align}\label{eq:psi12der}
\big|\p_x^m\psi_{12}^{(\d)}(\hat\bx, x)\big|\lesssim e^{-\varkappa_{|m|_\scale}|\bx|_\scale} 
\sum_j |x-x_j|^{-|m|_\scale} \1_{\{|x-x_j|>\d\}}(\hat\bx, x)
,\ \quad m\in\mathbb N_0^3,
\end{align}
uniformly in $\d>0$.

\begin{lem}\label{lem:psi12d} 
For any $l\ge 2$ the operator $\iop(\psi_{12}^{(\d)})a\1_{\CC_n}$ belongs to $\BS_{q, \infty}$ with
\begin{align}\label{eq:q}
\frac{1}{q} = \frac{1}{2} + \frac{l}{3}, 
\end{align}
and 
\begin{align}\label{eq:psi12d}
\|\iop(\psi_{12}^{(\d)})\, a\, \1_{\CC_n}\|_{q, \infty}
\lesssim e^{-\varkappa_{\scalel{l}}|n|_\scale}\d^{-l+\frac{3}{2}}\|a\|_{\plainL2(\CC_n)},
\end{align}
for all $\d\in (0, \d_0]$, 
with an implicit constant depending on $l$ and $\d_0$ only.
\end{lem}

\begin{proof} 
According to \eqref{eq:psi12der}, 
$\psi_{12}^{(\d)}(\hat\bx, \ \cdot\ )\in \plainH{l}(\CC_n)$ 
for a.e. $\hat\bx\in\R^{3N-3}$ with an arbitrary $l\ge 1$ and for $l\ge 2$ we have  
\begin{align*}
e^{2\varkappa_{\scalel{l}}|\hat\bx|_\scale}\|\psi^{(\d)}_{12}(\hat\bx, &\ \ \cdot\ )\|_{\plainH{l}}^2 
\lesssim e^{-2\varkappa_{\scalel{l}}|n|_\scale}\int_{\CC_n} 
\bigg( 1 +
\sum_j |x-x_j|^{-2l} \1_{\{|x-x_j|>\d\}}(\hat\bx, x) \bigg) dx \\[0.2cm]
\lesssim &\  e^{-2\varkappa_{\scalel{l}} |n|_\scale}\big(1 + \d^{-2l+3}\big)
\lesssim  e^{-2\varkappa_{\scalel{l}} |n|_\scale}\d^{-2l+3}.
\end{align*} 
Now the bound 
\eqref{eq:psi12d} follows from Proposition \ref{prop:BS} with $d = 3$ and $b(\hat\bx) = 1$.  
\end{proof}

\subsection{Partition of $\Psi_n$: step 2}
It is important to note that the right-hand side of 
\eqref{eq:psi12der} contains the factor $|x-x_j|^{-|m|_\scale}$ instead of $|x-x_j|^{1-|m|_\scale}$ 
that is present 
in \eqref{eq:FS1}. This is a consequence of the fact that 
the bound \eqref{eq:FS} holds for $|m|_1\ge 1$, but not for $m = 0$. 
As we see later on, in spite of this loss of one power of $|x-x_j|$,  the estimate \eqref{eq:psi12d} is 
sufficient for derivation of the sharp bounds \eqref{eq:psifullest} and \eqref{eq:psifull}. 
However, when considering the term $\psi_{11}^{(\d)}$ in \eqref{eq:split2} 
the bound by $|x-x_j|^{-|m|_\scale}$ is not enough, and 
we need to have the factor $|x-x_j|^{1-|m|_\scale}$, just as in \eqref{eq:FS1}. To achieve this we 
have to ``correct" the kernel $\psi_{11}^{(\d)}$ with the help of the auxiliary kernel
\begin{align*}
\eta^{(\d)}(\hat\bx, x) = \sum_{j} \t\big(|x-x_j|\d^{-1}\big) \psi_1^{(\d)}(\hat\bx, x_j).
\end{align*}
As the next lemma shows, the kernel $\eta^{(\d)}$ has properties similar to those of $\psi_{12}^{(\d)}$. 

\begin{lem}\label{lem:etad}
For any $l\ge 2$ the operator $\iop(\eta^{(\d)})a\1_{\CC_n}$ belongs to $\BS_{q, \infty}$ 
with the parameter $q$ defined in \eqref{eq:q}, 
and 
\begin{align}\label{eq:etad}
\|\iop(\eta^{(\d)})\, a\,\1_{\CC_n}\|_{q, \infty}\lesssim e^{-\varkappa_{\scalel{l}}|n|_\scale}
\d^{-l+\frac{3}{2}}\|a\|_{\plainL2(\CC_n)},
\end{align}
for all $\d\in (0, \d_0]$.
\end{lem}

\begin{proof} 
Using \eqref{eq:expon} and \eqref{eq:nu} we get 
\begin{align*}
|\p_x^m\eta^{(\d)}(\hat\bx, x)|\lesssim 
\d^{-|m|_\scale} e^{-\varkappa_{\scalel{l}}|n|_\scale-\varkappa_{\scalel{l}}|\hat\bx|_\scale}
\sum_{j}\1_{\{|x-x_j|<2\d\}}(\hat\bx, x),\ m\in\mathbb N_0^3, \ |m|_1\le l.
\end{align*}
Therefore, $\eta^{(\d)}(\hat\bx, \ \cdot\ )\in \plainH{l}(\CC_n)$ 
for a.e. $\hat\bx\in\R^{3N-3}$ with an arbitrary $l \ge 1$ and for $l\ge 2$ we have
\begin{align*}
e^{2\varkappa_{\scalel{l}}|\hat\bx|_\scale}\|\eta^{(\d)}(\hat\bx, &\ \ \cdot\ )\|_{\plainH{l}}^2 
\lesssim e^{-2\varkappa_{\scalel{l}} |n|_\scale}  \int_{\CC_n} 
\big( 1 + \d^{-2l}\sum_j \1_{\{|x-x_j|<2\d\}}(\hat\bx, x) \big) dx \\[0.2cm]
\lesssim &\ e^{-2\varkappa_{\scalel{l}}|n|_\scale}\big(1 + \d^{-2l+3}\big)
\lesssim  e^{-2\varkappa_{\scalel{l}} |n|_\scale}\d^{-2l+3}.
\end{align*} 
Now the required bound follows from Proposition \ref{prop:BS} with $d = 3$ and $b(\hat\bx) = 1$.  
\end{proof}

Let us now investigate the ``corrected" kernel $\psi_{11}^{(\d)}$, and consider instead of it 
the kernel 
\begin{align}\label{eq:split3}
\phi^{(\d)} = &\ \psi_{11}^{(\d)} - \eta^{(\d)} 
= \sum_{j} \phi_j^{(\d)},\\
\phi_j^{(\d)}(\hat\bx, x) = &\ \t\big(|x-x_j|\d^{-1}\big)\big(
\psi_1^{(\d)}(\hat\bx, x) - \psi_1^{(\d)}(\hat\bx, x_j)
\big). \notag
\end{align} 
Before proceeding to the next step of the construction, we estimate the difference 
$\psi_1^{(\d)}(\hat\bx, x) - \psi_1^{(\d)}(\hat\bx, x_j)$. 
It follows from \eqref{eq:FS} with $|m|_1=1$ that 
\begin{align}\label{eq:m01}
\big|\psi_1^{(\d)}(\hat\bx, x) - \psi_1^{(\d)}(\hat\bx, x_j)\big|
\le &\ |x-x_j| 
\max_{t\in[0,1]} |\nabla_x\psi_1^{(\d)}(\hat\bx, tx_j + (1-t) x)|\notag\\
\lesssim &\ |x-x_j|\, e^{-\varkappa_\scalel{1}|\bx|_\scale}\chi^{(\d)}(\hat\bx).
\end{align}
In order to estimate the derivatives of this difference, we make the following observation. 
By the definition of $\t$, we have $|x-x_j|<2\d$ on the support of $\phi_j^{(\d)}$. Furthermore, 
the balls $\{x\in\R^3: |x-x_j|< 2\d\}\subset\R^3$, $j=0, 1, \dots, N-1$, are pairwise disjoint since $\hat\bx\in\Om^{(\d)}$.
As a consequence, 
\begin{align*}
\dc(\hat\bx, x) = |x-x_j|,\quad \textup{if} \quad |x-x_j|<2\d,\ \hat\bx\in\Om^{(\d)}.
\end{align*}
Consequently, the bound \eqref{eq:FS} together with \eqref{eq:m01} lead to 
\begin{align}\label{eq:m0}
\big|\p_x^m\big(\psi_1^{(\d)}(\hat\bx, x) - &\ \psi_1^{(\d)}(\hat\bx, x_j)\big)\big|\notag\\
&\ \lesssim |x-x_j|^{1-|m|_\scale} e^{-\varkappa_{|m|_\scale}|\bx|_\scale}\chi^{(\d)}(\hat\bx), 
\quad \textup{if} \quad |x-x_j|<2\d,
\end{align}
for \underline{all} $m\in\mathbb N_0^3$. 
Here we have also used our convention that $\varkappa_0 = \varkappa_1$, see \eqref{eq:monotone}. 

Now return to the functions $\phi_j^{(\d)}$, 
see \eqref{eq:split3}. 
The $\phi_j^{(\d)}(\hat\bx, x)$ is again partitioned in the sum of two new kernels. 
At this (last) stage of the partition we introduce a new parameter $\varepsilon\le\d/2$. 
With this choice of $\varepsilon$ we have $\t(t\varepsilon^{-1}) = \t(t\varepsilon^{-1})\t(t\d^{-1})$, 
so that 
\begin{align*}
\phi_j^{(\d)} = \xi_j^{(\d, \varepsilon)} + \b_j^{(\d, \varepsilon)},\quad 
j = 0, 1, 2, \dots, N-1, 
\end{align*}
with
 \begin{align*}
\xi_j^{(\d, \varepsilon)}(\hat\bx, x) = &\ \t\big(|x-x_j|\varepsilon^{-1}\big)
\big(\psi_1^{(\d)}(\hat\bx, x) - \psi_1^{(\d)}(\hat\bx, x_j) 
\big),\\
\b_{j}^{(\d, \varepsilon)}(\hat\bx, x)
= &\ \t\big(|x-x_j|\d^{-1}\big)
\z\big(|x-x_j|\varepsilon^{-1}\big)
\big(\psi_1^{(\d)}(\hat\bx, x) - \psi_1^{(\d)}(\hat\bx, x_j) 
\big).
\end{align*} 
Therefore 
\begin{align}\label{eq:phidelsplit}
\phi^{(\d)}  = &\ \xi^{(\d, \varepsilon)} + \b^{(\d, \varepsilon)}, \quad \textup{where}\\
\xi^{(\d, \varepsilon)}  = &\ \sum_j\xi_j^{(\d, \varepsilon)},\quad 
\b^{(\d, \varepsilon)} = \sum_j \b_j^{(\d, \varepsilon)}.\notag
\end{align}
In the next lemma we introduce a weight $b\in\plainL\infty(\R^{3N-3})$. Recall 
that under this condition 
the integral $M^{(l)}(b)$ defined in \eqref{eq:mb} is finite for all $l\ge 1$. 

\begin{lem} \label{lem:xi}
Let $a\in\plainL2(\CC_n)$ and $b\in\plainL\infty(\R^{3N-3})$. Then 
$b\,\iop(\xi^{(\d, \varepsilon)})\, a\1_{\CC_n}\in \BS_{6/7, \infty}$ and 
\begin{align}\label{eq:phi11de}
\|b\,\iop(\xi^{(\d, \varepsilon)})\, a\1_{\CC_n}\|_{6/7, \infty}
\lesssim  e^{-\varkappa_{\scalel{2}}|n|_\scale}\varepsilon^{\frac{1}{2}} M^{(2)}(b) \|a\|_{\plainL2(\CC_n)},
\end{align}
for all $\varepsilon\in (0, 1]$ and $\d \in [2\varepsilon, 2]$. 
\end{lem}

\begin{proof} 
According to \eqref{eq:triangle}, it suffices to prove \eqref{eq:phi11de} for each $j = 0, 1, \dots, N-1,$ 
 individually. 
It follows from \eqref{eq:cutoff} that 
\begin{align*}
\big|\p_x^m\t(|x-x_j|\varepsilon^{-1})\big|\lesssim |x-x_j|^{-|m|_\scale}\1_{\{|x-x_j|<2\varepsilon\}}(\hat\bx, x),
\end{align*}
uniformly in $\varepsilon>0, \d> 2\varepsilon$,  
for all $m\in\mathbb N_0^3$. Together with \eqref{eq:m0} this implies that
\begin{align*}
|\p_x^m 
\xi_j^{(\d, \varepsilon)}(\hat\bx, x)| 
\lesssim |x-x_j|^{1-|m|_\scale} e^{-\varkappa_{\scalel{l}}|\bx|_\scale}
\1_{\{|x-x_j|< 2\varepsilon\}}(\hat\bx, x),
\end{align*}
for all $m\in\mathbb N_0^3$, $|m|_1\le l$.
Thus $\xi_j^{(\d, \varepsilon)}( \hat\bx,\ \cdot\ )\in \plainH2(\CC_n)$ 
for a.e. $\hat\bx\in\R^{3N-3}$ and 
\begin{align*}
e^{2\varkappa_{\scalel{2}}|\hat\bx|_\scale}\|\xi_j^{(\d, \varepsilon)}( \hat\bx &,\ \cdot\ )\|_{\plainH2}^2 
\lesssim e^{-2\varkappa_{\scalel{2}} |n|_\scale}\int_{\CC_n}
\big(1+ |x-x_j|^{-2}\big) \1_{\{|x-x_j|<2\varepsilon\}}
 dx \\[0.2cm]
\lesssim &\ e^{-2\varkappa_{\scalel{2}} |n|_\scale} (\varepsilon^3 + \varepsilon)
\lesssim e^{-2\varkappa_{\scalel{2}} |n|_\scale} \varepsilon.
\end{align*}  
It follows from Proposition \ref{prop:BS} with $l=2, d=3$ that 
$b\iop(\xi_j^{(\d, \varepsilon)}) a \1_{\CC_n}\in \BS_{6/7, \infty}$ and
\begin{align*}
\|b\iop(\xi_j^{(\d, \varepsilon)})a\,\1_{\CC_n}\|_{6/7, \infty}
\lesssim &\ e^{-\varkappa_{\scalel{2}} |n|_\scale}\bigg[\int_{\R^{3N-3}} |b(\hat\bx)|^2
\|\xi_j^{(\d, \varepsilon)}(\hat\bx,\ \cdot\ )\|_{\plainH2}^2 e^{-2\varkappa_{\scalel{2}} |\hat\bx|_\scale}
d\hat\bx \bigg]^{\frac{1}{2}} \|a\|_{\plainL2(\CC_n)}\\
\lesssim &\ e^{-\varkappa_{\scalel{2}} |n|_\scale} \varepsilon^{\frac{1}{2}}\, M^{(2)}(b)\,\|a\|_{\plainL2(\CC_n)}.  
\end{align*} 
This completes the proof of \eqref{eq:phi11de}. 
\end{proof}

\begin{lem} \label{lem:beta}
Let $a\in\plainL2(\CC_n)$ and $b\in \plainL\infty(\R^{3N-3})$. Then  
for any $l\ge 3$ the operator   
$b\, \iop(\b^{(\d, \varepsilon)}) a\1_{\CC_n}$ belongs to $\BS_{q, \infty}$ with 
the parameter $q$ defined in \eqref{eq:q}, and 
\begin{align}\label{eq:phi12de}
\| b\, \iop(\b^{(\d, \varepsilon)})\, a \1_{\CC_n} \|_{q, \infty}
\lesssim e^{-\varkappa_{\scalel{l}} |n|_\scale}
\varepsilon^{-l+\frac{5}{2}} M^{(l)}(b)\|a\|_{\plainL2(\CC_n)},
\end{align} 
for all $\varepsilon\in (0, 1]$ and $\d \in [2\varepsilon, 2]$. 
\end{lem}

\begin{proof} 
As in the previous lemma, due to \eqref{eq:triangle}, it suffices to prove 
\eqref{eq:phi12de} for each $j = 0, 1, \dots, N-1,$ individually. 
It follows from \eqref{eq:cutoff} that 
\begin{align*}
\big|\p_x^m\big(\t(|x-x_j|\d^{-1})\z(|x-x_j|\varepsilon^{-1})\big)\big|\lesssim |x-x_j|^{-|m|_\scale}
\1_{\{\varepsilon<|x-x_j|<2\d\}}(\hat\bx, x),
\end{align*}
uniformly in $\varepsilon>0, \d> 2\varepsilon$,  
for all $m\in\mathbb N_0^3$. Together with \eqref{eq:m0} this implies that
\begin{align*}
|\p_x^m \b_j^{(\d, \varepsilon)}(\hat\bx, x)|
\lesssim |x-x_j|^{1-|m|_\scale} e^{-\varkappa_{\scalel{l}}|\bx|_\scale}
\1_{\{|x-x_j|> \varepsilon\}}(\hat\bx, x),
\end{align*}
for all $m\in\mathbb N_0^3$, $|m|_1\le l$.
Thus $\b_j^{(\d, \varepsilon)}( \hat\bx,\ \cdot\ )\in \plainH{l}(\CC_n)$ 
for a.e. $\hat\bx\in\R^{3N-3}$ with an arbitrary $l\ge 1$ and for $l\ge 3$ we have 
\begin{align*}
e^{2\varkappa_{\scalel{l}}|\hat\bx|_\scale}\|\b_j^{(\d, \varepsilon)}( \hat\bx &,\ \cdot\ )\|_{\plainH{l}}^2 
\lesssim
e^{-2\varkappa_{\scalel{l}}|n|_\scale}
 \int_{\CC_n}
\big(1+ |x-x_j|^{2-2l}\big) \1_{\{|x-x_j|>\varepsilon\}}
 dx \\[0.2cm]
\lesssim &\ e^{-2\varkappa_{\scalel{l}} |n|_\scale} (1+ \varepsilon^{5-2l})
\lesssim e^{-2\varkappa_{\scalel{l}} |n|_\scale}\varepsilon^{5-2l}.
\end{align*} 
Using Proposition \ref{prop:BS}  with $d = 3$ and arbitrary $l\ge 3$, we get
that $b\,\iop(\b_j^{(\d, \varepsilon)})a\1_{\CC_n}\in \BS_{q, \infty}$ and
\begin{align*}
\| b\,\iop(\b_j^{(\d, \varepsilon)}) a\1_{\CC_n}\|_{q, \infty}
\lesssim &\ e^{-\varkappa_{\scalel{l}}|n|_\scale}\bigg[\int_{\R^{3N-3}} |b(\hat\bx)|^2
\|\b_j^{(\d, \varepsilon)}(\hat\bx,\ \cdot\ )\|_{\plainH{l}}^2 e^{-2\varkappa_{\scalel{l}}|\hat\bx|_\scale}
d\hat\bx \bigg]^{\frac{1}{2}} \|a\|_{\plainL2(\CC_n)}\\
\lesssim &\ e^{-\varkappa_{\scalel{l}}|n|_\scale} 
\varepsilon^{-l+\frac{5}{2}}\,M^{(l)}(b)\,\|a\|_{\plainL2(\CC_n)}.  
\end{align*}
This completes the proof of \eqref{eq:phi12de}. 
\end{proof}

\section{Proof of Theorems \ref{thm:psifull} and \ref{thm:main}}\label{sect:proofs}

Her we put together the estimates obtained in the previous section to complete the proof 
of Theorem \ref{thm:psifull}. Recall again that the quantities $S^{(l)}_q(a)$  and $M^{(l)}(b)$  
are defined in \eqref{eq:Sq} and \eqref{eq:mb} respectively.

\begin{lem} 
Suppose that $b\in\plainL\infty(\R^{3N-3})$ 
and $a\in\plainL2(\CC_n)$.
Then  $ b\Psi_n a\in \BS_{3/4, \infty}$ and 
\begin{align}
\| b\Psi_n a\|_{3/4, \infty} \lesssim &\ e^{-\varkappa_{\scalel{3}}|n|_{\scale}}
\|b\|_{\plainL\infty} \|a\|_{\plainL2(\CC)},
\label{eq:psinorm}\\
 \SfG_{3/4}(b\Psi_n a)
\lesssim  &\ \big(e^{-\varkappa_{\scalel{4}} |n|_\scale}M^{(4)}(b)\|a\|_{\plainL2(\CC_n)}\big)^{\frac{3}{4}},
\label{eq:psin}
\end{align}
for all $n\in \mathbb Z^3$.
\end{lem}
   
\begin{proof} 
Now we can put together all the estimates for the singular numbers, 
obtained above. Without loss of generality assume that $\|b\|_{\plainL\infty}\le 1$ and 
$\|a\|_{\plainL2(\CC)}\le 1$.

By \eqref{eq:split1}, \eqref{eq:split2}, \eqref{eq:split3} 
and \eqref{eq:phidelsplit} we have 
\begin{align*}
\psi = \xi^{(\d, \varepsilon)} + \b^{(\d, \varepsilon)}
+ \eta^{(\d)} + 
\psi_{12}^{(\d)} + \psi_2^{(\d)}.  
\end{align*}
According to \eqref{eq:psidelcube},\eqref{eq:phi11de}, and the inequality \eqref{eq:triangle},
\begin{align*}
\|b\,\iop(\xi^{(\d, \varepsilon)}+\psi_2^{(\d)})a\1_{\CC_n}\|_{6/7, \infty}^{6/7}
\le &\ 7\big(\|b\,\iop(\xi^{(\d, \varepsilon)})a\1_{\CC_n}\|_{6/7, \infty}^{6/7} 
+ \|\iop(\psi_2^{(\d)})a\1_{\CC_n}\|_{6/7, \infty}^{6/7}\big)\\
\lesssim  &\ e^{-6\varkappa_{\scalel{2}}|n|_\scale/7}\big(\varepsilon^{\frac{1}{2}} M^{(2)}(b) + \d^{\frac{3}{2}}
\big)^{6/7},  
\end{align*}
so that, by definition \eqref{eq:quasi},
\begin{align}\label{eq:sw}
s_k\big(b\,\iop(\xi^{(\d, \varepsilon)}+&\ \psi_2^{(\d)})a\1_{\CC_n}\big)\notag\\[0.2cm]
\lesssim &\ 
e^{-\varkappa_{\scalel{2}}|n|_\scale}\big(\varepsilon^{\frac{1}{2}} M^{(2)}(b) + \d^{\frac{3}{2}}
\big)   
\ k^{-\frac{7}{6}},\quad k = 1, 2, \dots.
\end{align}
Similarly, using \eqref{eq:psi12d}, \eqref{eq:etad} and \eqref{eq:phi12de} with 
one and the same $l\ge 3$, we obtain that 
\begin{align*}
\|b\,\iop(\b^{(\d, \varepsilon)}+\eta^{(\d)}+\psi_{12}^{(\d)})a\1_{\CC_n}\|_{q, \infty}
\lesssim &\ e^{-\varkappa_{\scalel{l}}|n|_\scale} 
\big(\varepsilon^{-l+\frac{5}{2}} M^{(l)}(b) + \d^{-l+\frac{3}{2}}\big),  
\end{align*} 
with $1/q = 1/2 + l/3$
and hence, 
\begin{align}\label{eq:ss}
s_k\big(b\,\iop(\b^{(\d, \varepsilon)}+&\ \eta^{(\d)}+\psi_{12}^{(\d)})a\1_{\CC_n}\big)\notag\\
\lesssim &\ e^{-\varkappa_{\scalel{l}}|n|_\scale} 
\big(\varepsilon^{-l+\frac{5}{2}} M^{(l)}(b) + \d^{-l+\frac{3}{2}}\big)  
k^{-\frac{1}{2} - \frac{l}{3}},\quad k = 1, 2, \dots.
\end{align}
Due to \eqref{eq:2k} and \eqref{eq:monotone}, combining \eqref{eq:sw} and \eqref{eq:ss}, we get the estimate
\begin{align}\label{eq:2kminus}
s_{2k}(b\Psi_n a)\le &\ s_{2k-1}(b\Psi_n a)\notag\\
\lesssim &\ e^{-\varkappa_{\scalel{l}} |n|_\scale}\bigg[\big(\varepsilon^{\frac{1}{2}} 
M^{(l)}(b) + \d^{\frac{3}{2}}\big)k^{-\frac{7}{6}}
+ \big(\varepsilon^{-l+\frac{5}{2}} M^{(l)}(b) + \d^{-l+\frac{3}{2}}\big)k^{-\frac{1}{2} - \frac{l}{3}}
\bigg],  
\end{align}
where we have used that $M^{(2)}(b)\le M^{(l)}(b)$.
Rewrite the expression in the square brackets, gathering 
the terms containing $\varepsilon$ and $\d$ in two different groups:
\begin{align*}
\big(\varepsilon^{\frac{1}{2}} &\ M^{(l)}(b) k^{-\frac{7}{6}} 
+  \varepsilon^{-l+\frac{5}{2}} M^{(l)}(b) k^{-\frac{1}{2} - \frac{l}{3}} \big)
+ \big(\d^{\frac{3}{2}} k^{-\frac{7}{6}} + \d^{-l+\frac{3}{2}}k^{-\frac{1}{2} - \frac{l}{3}}\big)\notag\\
= &\ 
\varepsilon^{\frac{1}{2}} M^{(l)}(b) k^{-\frac{7}{6}}\big( 1 
+  \varepsilon^{-l+2} k^{\frac{2-l}{3}} \big)
+ \d^{\frac{3}{2}} k^{-\frac{7}{6}} \big(1 + \d^{-l}k^{\frac{2-l}{3}}\big).
\end{align*}
Since $\varepsilon\in (0,1]$ 
and $\d\in [2\varepsilon, 2]$ are arbitrary, we can pick $\varepsilon = \varepsilon_k = k^{-1/3}$ and 
$\d = \d_k = 4k^{2/(3l) -1/3}$, so that  the condition 
$\d>2\varepsilon$ is satisfied for all $k = 1, 2, \dots$, and 
\begin{align*}
\varepsilon^{-l+2} k^{\frac{2-l}{3}}  = &\ 1,\quad 
\varepsilon^{\frac{1}{2}} k^{-\frac{7}{6}} = k^{-\frac{4}{3}},\\
\d^{-l}k^{\frac{2-l}{3}} = &\ 4^{-l}, 
\quad \d^{\frac{3}{2}} k^{-\frac{7}{6}} = 4^{\frac{3}{2}}k^{\frac{1}{l} - \frac{5}{3}}.
\end{align*}
Thus the bound \eqref{eq:2kminus} rewrites as
\begin{align}\label{eq:transform}
s_{2k}(b\Psi_n a)\le s_{2k-1}(b\Psi_n a)\le e^{-\varkappa{\scalel{l}}|n|_{\scale}}
\big(M^{(l)}(b)k^{-\frac{4}{3}} + k^{\frac{1}{l} - \frac{5}{3}}\big). 
\end{align}
Using the bound $M^{(l)}(b)\lesssim \|b\|_{\plainL\infty}\le 1$, and taking $l=3$ we conclude that 
\begin{align*}
s_k(b\Psi_n a)\lesssim e^{-\varkappa_{\scalel{3}}|n|_{\scale}} 
k^{-\frac{4}{3}}. 
\end{align*}
This leads to \eqref{eq:psinorm}. 

In order to obtain \eqref{eq:psin}, we use \eqref{eq:transform} to write
\begin{align*}
\limsup_{k\to\infty}k^{\frac{4}{3}} s_{k}\big(b\Psi_n a\big)
\lesssim  
e^{-\varkappa_{\scalel{l}}|n|_\scale} \limsup_{k\to \infty}\big(M^{(l)}(b) 
+ k^{\frac{1}{l}-\frac{1}{3}}\big).
\end{align*} 
Taking $l = 4$ we ensure that the second term in the brackets tends to zero. 
Therefore,
\begin{align*}
\limsup_{k\to\infty}k^{\frac{4}{3}} s_{k}\big(b\Psi_n a\big)
\lesssim 
e^{-\varkappa_{\scalel{4}} |n|_\scale} M^{(4)}(b).
\end{align*} 
Applying definition \eqref{eq:limsupinf}, we arrive at \eqref{eq:psin}.
\end{proof}

\begin{proof}[Proof of Theorems \ref{thm:psifull} and \ref{thm:main}] 
Since $\Psi = \sum_{n\in\Z^3} \Psi_n$, we have, 
by \eqref{eq:triangle} and \eqref{eq:psinorm},
\begin{align*}
\|b\Psi a\|_{3/4, \infty}^{3/4}\le &\ 4 \sum_{n\in\Z^3} \|b\Psi_n a\|_{3/4, \infty}^{3/4}\\
\lesssim &\ \|b\|_{\plainL\infty}^{\frac{3}{4}}\sum_{n\in\Z^3} e^{-\frac{3}{4}\varkappa_{\scalel{3}} |n|_\scale} 
\|a\|_{\plainL2(\CC_n)}^{\frac{3}{4}}
= \|b\|_{\plainL\infty}^{\frac{3}{4}} \big(S_{3/4}^{(3)}(a)\big)^{\frac{3}{4}}
<\infty.
\end{align*}
This proves \eqref{eq:psifullest}. 

To prove \eqref{eq:psifull} we use Lemma \ref{lem:triangleg}. According to \eqref{eq:triangleg} 
and \eqref{eq:psin}, 
\begin{align*}
\SfG_{3/4} (b\Psi a)\le &\ 4 \sum_{n\in\Z^3} \SfG_{3/4}(b\Psi_n a)\\
\lesssim &\ \big(M^{(4)}(b)\big)^{\frac{3}{4}}\sum_{n\in\Z^3} e^{-\frac{3}{4}\varkappa_{\scalel{4}} |n|_\scale} 
\|a\|_{\plainL2(\CC_n)}^{\frac{3}{4}}
= \big(M^{(4)}(b)\big)^{\frac{3}{4}} \big(S_{3/4}^{(4)}(a)\big)^{\frac{3}{4}}
<\infty.
\end{align*}
This completes the proof of Theorem \ref{thm:psifull}.

Using \eqref{eq:psifullest} with $a(x) = 1$ and $b(\hat\bx) = 1$ we get 
$\|\Psi\|_{3/4, \infty}<\infty$, which implies 
that $s_k(\Psi)\lesssim k^{-4/3}$, and hence $\l_k(\G) = s_k(\Psi)^2\lesssim k^{-8/3}$. This proves 
Theorem \ref{thm:main}.
\end{proof}     

\textbf{Acknowledgments.} The author is grateful to S. Fournais, T. Hoffmann-Ostenhof, 
M. Lewin and T. \O. S\o rensen
for stumulating discussions and advice. 

The author was supported by the EPSRC grant EP/P024793/1.

\bibliographystyle{../beststyle}

\end{document}